\documentclass[11pt]{article}
\usepackage[margin=1in]{geometry}
\usepackage[T1]{fontenc}
\usepackage[utf8]{inputenc}
\usepackage{lmodern}
\usepackage{microtype}
\usepackage{mathtools,amssymb,amsthm,mathrsfs}
\usepackage{graphicx}
\usepackage{enumitem}
\usepackage{algorithm}
\usepackage{algpseudocode}
\usepackage{xcolor}
\usepackage{hyperref}
\usepackage[nameinlink,noabbrev]{cleveref}

\hypersetup{
  colorlinks=true,
  linkcolor=blue!55!black,
  citecolor=blue!55!black,
  urlcolor=blue!55!black
}

\newcommand{\R}{\mathbb{R}}

\newcommand{\range}{\mathrm{range}}
\newcommand{\Null}{\mathrm{Null}}

\newcommand{\ip}[2]{\langle #1,#2\rangle}
\newcommand{\norm}[1]{\lVert #1\rVert}

\newcommand{\ball}[2]{\mathbb{B}_{#1}(#2)}
\newcommand{\cO}{\mathcal{O}}

\newtheorem{theorem}{Theorem}[section]

\newtheorem{lemma}[theorem]{Lemma}

\theoremstyle{definition}
\newtheorem{assumption}[theorem]{Assumption}
\newtheorem{remark}[theorem]{Remark}

\numberwithin{equation}{section}

\title{A short proof of near-linear convergence of adaptive gradient descent\\under fourth-order growth and convexity}
\author{Damek Davis\thanks{Department of Statistics and Data Science, The Wharton School, University of Pennsylvania, Philadelphia, PA 19104; \texttt{damek@wharton.upenn.edu}. Research supported by an Alfred P.\ Sloan research fellowship and NSF DMS award 2047637.}
\and Dmitriy Drusvyatskiy\thanks{Halicio\u{g}lu Data Science Institute, University of California, San Diego, La Jolla, CA 92093; \texttt{ddrusv@ucsd.edu}. Research supported by NSF DMS-2306322, NSF CCF 1740551, and AFOSR FA9550-24-1-0092.}}
\date{}

\begin{document}
\maketitle

\begin{abstract}
Davis, Drusvyatskiy, and Jiang~\cite{DDJ} showed that gradient descent with an adaptive stepsize converges locally at a nearly-linear rate for smooth functions that grow at least quartically away from their minimizers. The argument is intricate, relying on monitoring the performance of the algorithm relative to a certain manifold of slow growth---called the \emph{ravine}. In this work, we provide a direct Lyapunov-based argument that bypasses these difficulties when the objective is in addition convex and a has a unique minimizer. As a byproduct of the argument, we obtain a more adaptive variant than the original algorithm with encouraging numerical performance.
\end{abstract}

\section{Introduction}

The recent paper~\cite{DDJ} proved that gradient descent with an adaptive stepsize converges locally at a nearly-linear rate for functions that grow at least quartically away from their solution set, even when the Hessian $H:=\nabla^2 f(0)$ is singular.
Their proof is organized around the \emph{ravine}, a smooth manifold containing the minimizer and along which $f$ behaves like a pure fourth power of the distance. The Hessian $H:=\nabla^2 f(0)$ splits the space: along its range, $f$ is locally quadratic; along its null space, growth is only quartic. Constant length gradient steps quickly drive the iterates toward the ravine by shrinking the quadratic part of~$f$. Once the iterates are sufficiently close to the ravine, a single Polyak step contracts the distance to the minimizer. Thus the algorithm interleaves gradient steps (to approach the ravine) with Polyak steps (to contract towards the minimizer). The proof of nearly-linear convergence is sophisticated, and is based on carefully monitoring the interaction between the iterates and the ravine.

In this work, we provide a direct and elementary proof of convergence under two additional assumptions: the minimizer is isolated and the function is convex. The Lyapunov argument we present bypasses the ravine construction entirely. Note that convexity does not trivialize the geometry of the ravine itself; this manifold can still be highly nonlinear. For instance, the function
\[
f(v,u)=\tfrac12(v+u^4)^2+u^4
\]
is convex near the origin with $\nabla^2 f(0)=\operatorname{diag}(1,0)$, but the ravine is the curve $v=-u^4$, not a subspace (\Cref{fig:ravine}).

\begin{figure}[ht]
\centering
\includegraphics[width=0.48\textwidth]{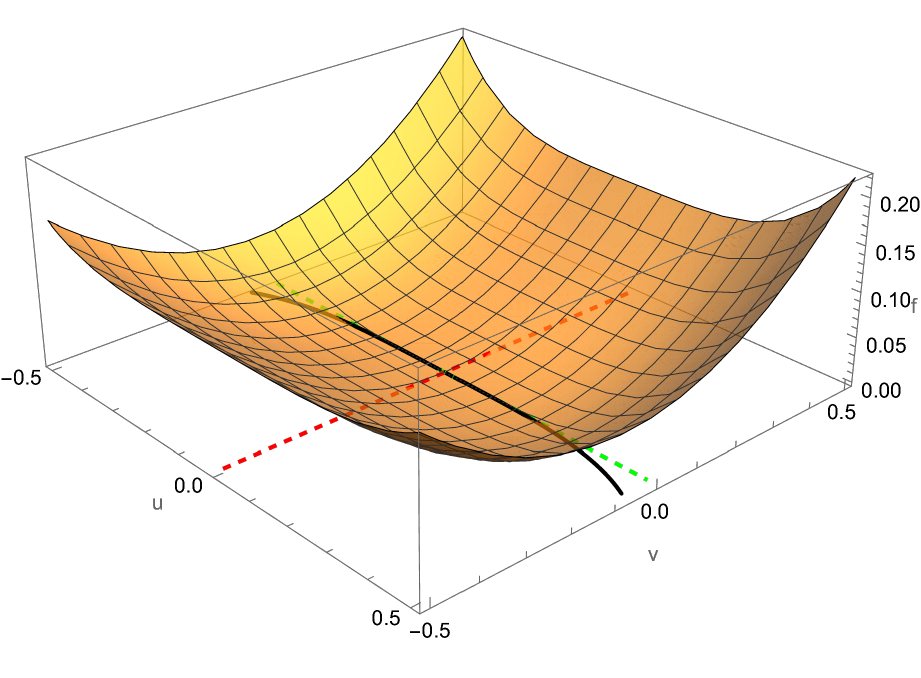}\hfill
\includegraphics[width=0.48\textwidth]{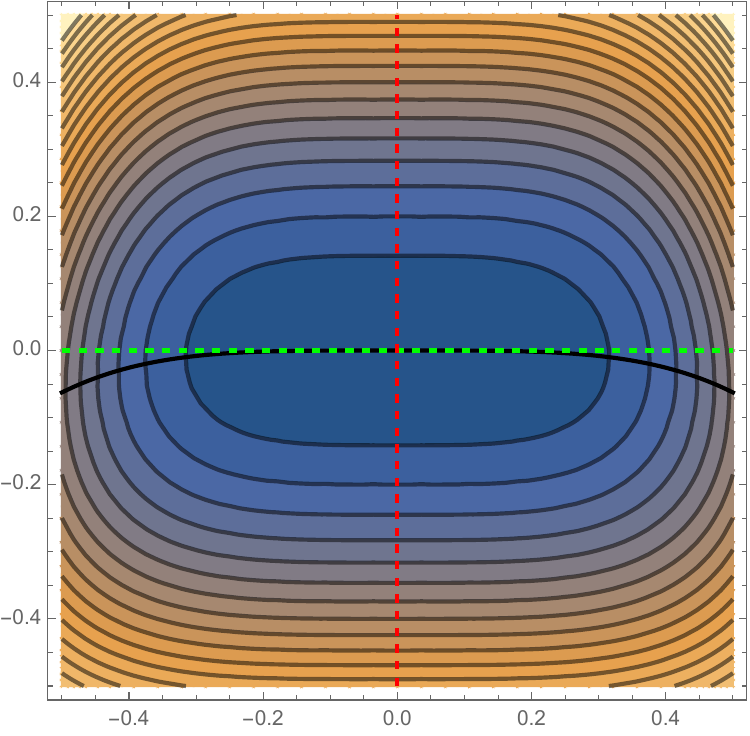}
\caption{The function $f(v,u)=\frac12(v+u^4)^2+u^4$ near the origin. \emph{Left:} surface plot. \emph{Right:} level sets. The solid black curve is the ravine $v=-u^4$; the dashed green line is the null space $\mathsf Q=\{v=0\}$; the dashed red line is the range $\mathsf P=\{u=0\}$.}
\label{fig:ravine}
\end{figure}

The simplification is instead in the proof mechanism. Setting the stage, write $\mathsf Q:=\Null(H)$, $\mathsf P:=\range(H)$, and let $P$ denote the orthogonal projection onto $\mathsf P$. Under constant-stepsize gradient descent, we show that the squared projected gradient $G(x):=\norm{P\nabla f(x)}^2$ contracts at a linear rate $1-\Theta(\eta)$ per step, until the iterate enters a region where the function behaves like a pure fourth power of the distance. Compared to the proof in the general case, which monitors the ravine explicitly, the projected-gradient contraction drives the iterate into the fourth-power regime directly. This new argument also simplifies the algorithm. Instead of the inner-outer loop structure in~\cite{DDJ}---where gradient and Polyak steps are scheduled in blocks of prescribed length---the algorithm (\Cref{alg:main}) runs gradient descent with a fixed stepsize $\eta$ and switches to a Polyak step whenever the ratio $R(x):=f(x)/\norm{\nabla f(x)}^{4/3}$ exceeds a threshold $\tau$, signaling that the iterate is in the fourth-power regime. When $R(x)\ge\tau$, the Polyak step contracts the distance by a factor $\sqrt{1-\tau^{3/2}\sqrt{m_0}}$.

The essential role of convexity is to eliminate the quadratic coupling $P\nabla^3 f(0)[u,u]$ in the Taylor expansion of the projected gradient. Writing $x=u+v$ with $u\in\mathsf Q$ and $v\in\mathsf P$, the general expansion reads
\[
P\nabla f(x)=Hv+\tfrac12 P\nabla^3 f(0)[u,u]+O\bigl(\norm{u}^3+\norm{u}\norm{v}+\norm{v}^2\bigr).
\]
Convexity forces $P\nabla^3 f(0)[u,u]=0$ for all $u\in\mathsf Q$ (\Cref{lem:vanishing}\ref{item:convex-kills-b}), so the projected gradient reduces to $Hv$ plus higher-order terms, and gradient steps in the fixed $\mathsf P/\mathsf Q$ splitting drive $\norm{v}$ to $O(\norm{u}^3)$. Once there, quartic aiming $\ip{\nabla f(x)}{x}\approx 4f(x)$ holds due to fourth-order behavior. Without the cubic vanishing, gradient steps still drive $\norm{v}$ to $O(\norm{u}^2)$, but the ravine curves away from $\mathsf Q$ at the same scale, so the fixed $\mathsf P/\mathsf Q$ splitting cannot separate tangential from normal dynamics. The original argument~\cite{DDJ} instead of $P$ uses a projection adapted to the ravine's normal bundle and works in coordinates that track this curvature. Replacing $u^4$ by $u^2$ in the earlier example illustrates the contrast: $g(v,u):=\tfrac12(v+u^2)^2+u^4$ has $\nabla^2 g(0)=\operatorname{diag}(1,0)$ and satisfies fourth-order growth, but $P\nabla^3 g(0)[e_u,e_u]=2e_v\neq 0$ and the ravine is the parabola $v=-u^2$ (\Cref{fig:ravine-nonconvex}).

\begin{figure}[ht]
\centering
\includegraphics[width=0.48\textwidth]{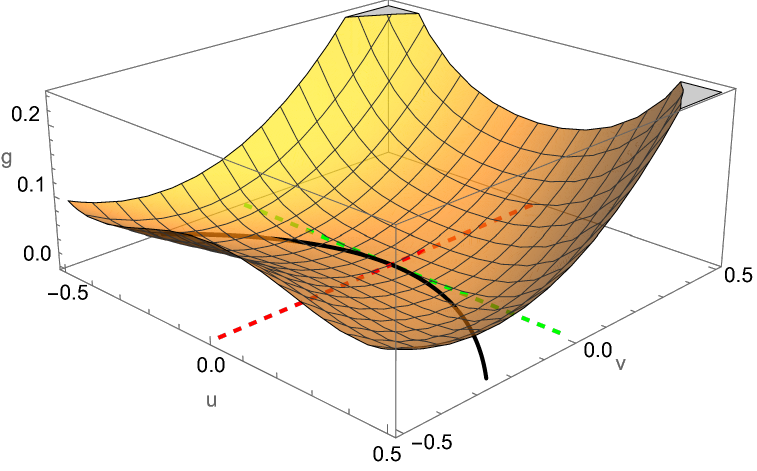}\hfill
\includegraphics[width=0.48\textwidth]{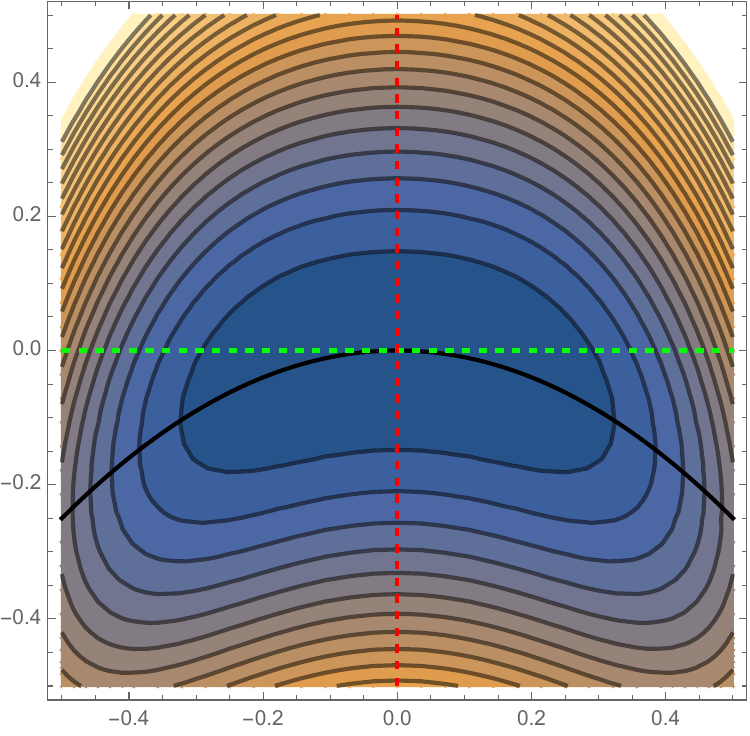}
\caption{The function $g(v,u)=\frac12(v+u^2)^2+u^4$ near the origin. \emph{Left:} surface plot. \emph{Right:} level sets. The solid black curve is the ravine $v=-u^2$; the dashed green line is $\mathsf Q=\{v=0\}$; the dashed red line is $\mathsf P=\{u=0\}$. Compare with \Cref{fig:ravine}: the ravine here has $O(\norm{u}^2)$ curvature, matching the nonvanishing cubic coupling $P\nabla^3 g(0)[e_u,e_u]=2e_v$.}
\label{fig:ravine-nonconvex}
\end{figure}

From the argument it is clear that one could replace convexity with the weaker hypothesis $P\nabla^3 f(0)[u,u]=0$ for all $u\in\mathsf Q$. We assume convexity because it is a more natural condition and because it additionally gives $\ip{\nabla f(x)}{x}\ge f(x)$, implying $\norm{x^+}\le\norm{x}$ for every iterate. This distance monotonicity trivially keeps all iterates near the minimizer---a further simplification.

\Cref{fig:intro-experiments} previews the practical performance of Algorithm~\ref{alg:main} on both functions $f$ and~$g$, compared with fixed-step gradient descent, pure Polyak, and the block-scheduled \texttt{GDPolyak} from~\cite{DDJ}. All hyperparameters---$(\eta,\tau)$ for Algorithm~\ref{alg:main} and the GD stepsize and block length for \texttt{GDPolyak}---were tuned by grid search. On both problems, Algorithm~\ref{alg:main} and \texttt{GDPolyak} converge to distance $10^{-6}$ in roughly $66$--$84$ iterations, while GD and Polyak plateau orders of magnitude above the target. The convex and nonconvex quartics behave nearly identically: the same tuned stepsize $\eta=1$ works for both. The tuned \texttt{GDPolyak} uses block length~$1$---alternating a single GD step with a single Polyak step---and both converging methods require $\eta\approx 1$.

\begin{figure}[ht]
\centering
\includegraphics[width=\textwidth]{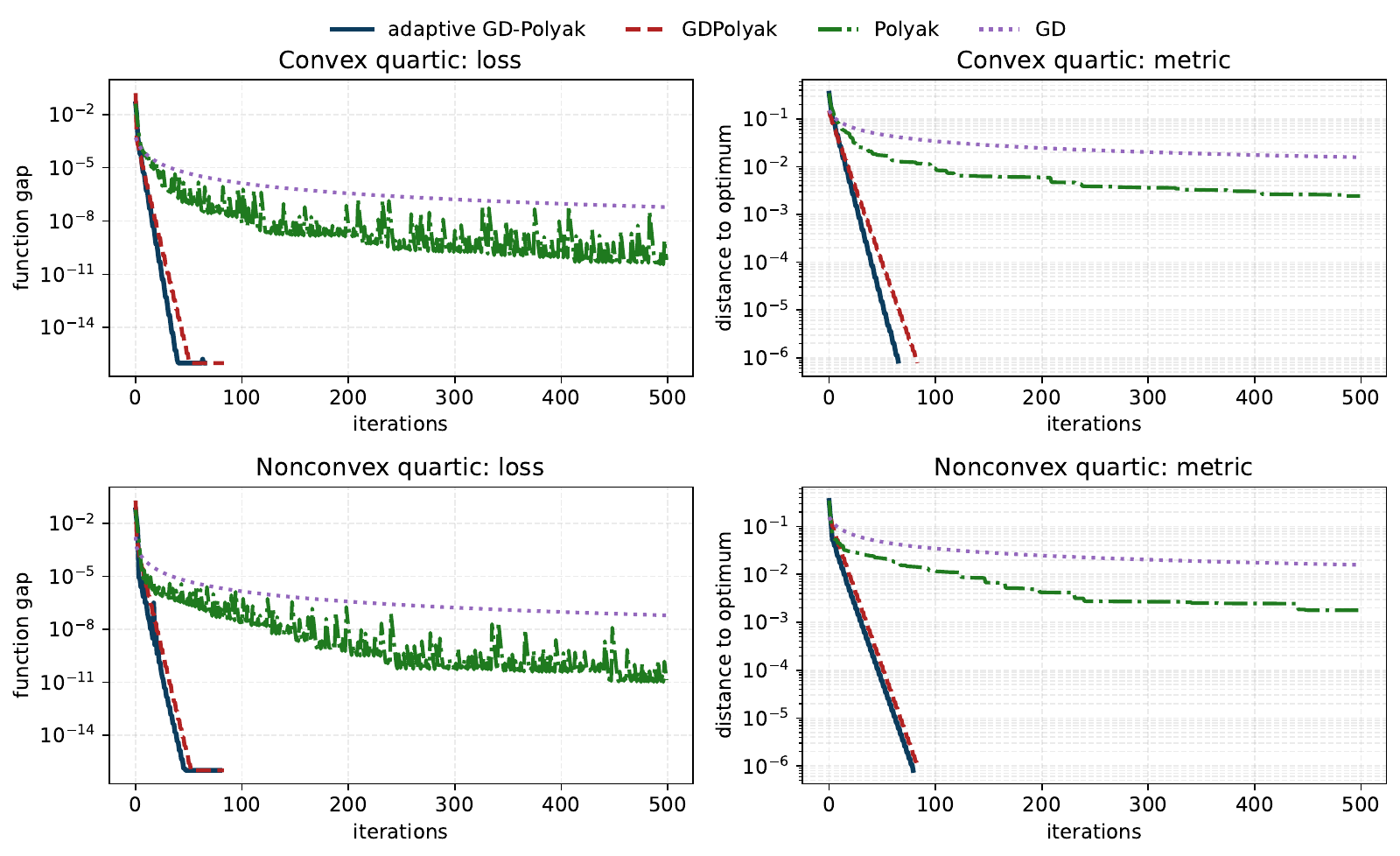}
\caption{Algorithm~\ref{alg:main} on the two quartics from \Cref{fig:ravine,fig:ravine-nonconvex}. Top: convex $f$ with $(\eta,\tau)=(1,0.15)$; $66$ iterations. Bottom: nonconvex $g$ with $(\eta,\tau)=(1,0.12)$; $80$ iterations. \texttt{GDPolyak} (stepsize~$1$, block length~$1$, i.e., alternating GD/Polyak every iteration): $84$ iterations on both. GD stepsize~$1$. Distance threshold~$10^{-6}$.}
\label{fig:intro-experiments}
\end{figure}

\paragraph{Related work.}
Classical results establish that gradient descent converges linearly for smooth functions satisfying quadratic growth~\cite{Polyak63,LuoTseng93,DrusvyatskiyLewis18,KarimiNutiniSchmidt16,NecNesGli19}. The Polyak stepsize $\eta_k=\bigl(f(x_k)-f_\star\bigr)/\norm{\nabla f(x_k)}^2$ was introduced in~\cite{Polyak87} and has been revisited recently in~\cite{HazanKakade19}. When growth is weaker than quadratic, constant-stepsize gradient descent converges only sublinearly~\cite{AttouchBolteSvaiter13}.

The geometric idea of exploiting a manifold of slow growth dates to the \emph{ravine method} of Gelfand and Tsetlin~\cite{GelfandTsetlin61}, which later influenced Polyak's heavy ball method~\cite{Polyak64} and Nesterov's accelerated gradient~\cite{Nesterov83}; see~\cite{AttouchFadili22} for a modern dynamical-systems perspective. In nonsmooth optimization, analogous structures appear as \emph{active manifolds} and \emph{partial smoothness}~\cite{Lewis02,DrusvyatskiyLewis14}, and the Normal Tangent Descent algorithm of~\cite{DavisJiang24} exploits such structure to obtain nearly-linear convergence for nonsmooth functions with quadratic growth.

A separate line of work shows that nonconstant stepsize schedules can accelerate gradient descent even for smooth strongly convex functions. Young~\cite{Young53} used Chebyshev-polynomial stepsizes for quadratics; more recent results include the ``big-step-little-step'' schedule of~\cite{KelnerEtAl22} for multiscale objectives, long steps for smooth convex functions~\cite{Grimmer24}, and the silver stepsize schedule~\cite{AltschulerParrilo23a,AltschulerParrilo23b}. Our algorithm's switching between constant gradient steps and a single Polyak step is closer in spirit to the epoch-based schedule in~\cite{DDJ}.

Key applications of fourth-order growth arise in rank-overparameterized matrix sensing and factorization~\cite{BurerMonteiro03,BurerMonteiro05,ChiLuChen19,ZhuoEtAl21} and overparameterized neural network training~\cite{XuDu23}. In the exact-rank regime, the objective grows quadratically and standard methods converge linearly~\cite{TuEtAl16,GeJinZheng17,ZhuEtAl18}. With rank overparameterization, growth degrades to quartic, and constant-stepsize gradient descent becomes sublinear~\cite{ZhuoEtAl21}; the GDPolyak algorithm of~\cite{DDJ} was the first method to achieve nearly-linear convergence in this setting.

\section{Setup, algorithm, and main theorem}

After subtracting the minimal value and translating the minimizer to the origin, we impose the following assumption.

\begin{assumption}[Fourth-order growth]\label{ass:main}
Let $f:\R^d\to\R$ be $C^4$ in a neighborhood of the origin with $f(0)=0$ and $\nabla f(0)=0$. Assume:
\begin{enumerate}[label=\textnormal{(A\arabic*)}, leftmargin=2.3em]
    \item \label{ass:quartic} There exist constants $m_0>0$ and $\rho_0>0$ such that
    \[
    f(x)\ge m_0\norm{x}^4
    \qquad\text{for all } \norm{x}\le \rho_0.
    \]
    \item \label{ass:singular} The Hessian $H:=\nabla^2 f(0)$ is singular and nonzero.
\end{enumerate}
\end{assumption}

Because the origin is a minimizer, $H$ is symmetric positive semidefinite. Set
\[
\mathsf P:=\range(H),\qquad \mathsf Q:=\Null(H),
\]
and let $P$ and $Q$ denote the orthogonal projections onto $\mathsf P$ and $\mathsf Q$, respectively. For $x\in\R^d$, write
\[
u:=Qx,\qquad v:=Px,
\qquad\text{so that } x=u+v,\quad u\perp v.
\]
Throughout, asymptotic notation ($O$, $\Theta$, $\Omega$) refers to the regime $\norm{x}\to0$. Define the ratio
\[
R(x):=\frac{f(x)}{\norm{\nabla f(x)}^{4/3}}.
\]

\begin{algorithm}[ht]
\caption{Adaptive GD--Polyak}\label{alg:main}
\begin{algorithmic}[1]
\Require initial point $x_0$, stepsize $\eta>0$, trigger threshold $\tau>0$
\For{$k=0,1,2,\dots$}
    \If{$\nabla f(x_k)=0$}
        \State \textbf{stop}
    \EndIf
    \If{$R(x_k)\ge \tau$}
        \State $x_{k+1}\gets x_k-\dfrac{f(x_k)}{\norm{\nabla f(x_k)}^2}\nabla f(x_k)$ \Comment{Polyak step}
    \Else
        \State $x_{k+1}\gets x_k-\eta\,\nabla f(x_k)$ \Comment{gradient step}
    \EndIf
\EndFor
\end{algorithmic}
\end{algorithm}

The following is the main theorem of the paper.
\begin{theorem}[Local convergence]\label{thm:main}
Suppose that Assumption~\ref{ass:main} holds and that $f$ is convex on a neighborhood of the origin. Then there exist constants $\rho,\eta_0,\tau_\star>0$ such that for every $\eta\in(0,\eta_0]$ and $x_0\in\ball{\rho}{0}$, Algorithm~\ref{alg:main} with $\tau=\tau_\star$ produces iterates satisfying $\norm{x_k}\to 0$. More precisely, for every $\varepsilon\in(0,\norm{x_0}]$, the algorithm reaches $\ball{\varepsilon}{0}$ after at most
\[
\cO\!\left(\eta^{-1}\log\frac{1}{\varepsilon}\log\frac{\norm{x_0}}{\varepsilon}\right)
\]
gradient and function evaluations, where $\cO(\cdot)$ hides dependence on $m_0,\rho_0,\rho,\eta_0, \tau_\star$ but not the dimension $d$.
\end{theorem}

\begin{remark}[Unknown optimal value]\label{rem:unknown-fstar}
Algorithm~\ref{alg:main} uses $f_\star$ in the Polyak step and the trigger $R$. In practice, the exact optimal value is rarely known, but a lower bound $\hat f_0\le f_\star$ is often available---for instance, $\hat f_0=0$ for any nonnegative loss. Following Algorithm~2 and Theorem~5.3 of~\cite{DDJ}, such a lower bound suffices: wrap Algorithm~\ref{alg:main} in an outer loop indexed by $j=1,\dots,J$, running the inner algorithm with $\hat f_j$ in place of $f_\star$ and the Polyak stepsize scaled by $1/2$, then updating $\hat f_{j+1}\gets (\hat f_j+\min_k f(x_k))/2$. Each outer iteration halves the gap $|f_\star-\hat f_j|$, and the Polyak contraction in Step~1 holds with modified constants whenever $\hat f_j\le f_\star$. Thus $\cO(\log(|f_\star-\hat f_0|/\varepsilon))$ outer iterations suffice, adding one logarithmic factor to the total complexity.
\end{remark}

\section{Local estimates}\label{sec:local}

Two lemmas supply all the Taylor information used in the proof. The first records vanishing of certain cubic terms; the second collects the gradient bounds.

\begin{lemma}[Vanishing cubics]\label{lem:vanishing}
For every $u\in\mathsf Q$ the following are true.
\begin{enumerate}[label=\textnormal{(\alph*)}, leftmargin=2.2em]
\item \label{item:flat-cubic} Equality $\nabla^3 f(0)[u,u,u]=0$ holds.
\item \label{item:convex-kills-b} If $f$ is convex near the origin, then $\nabla^3 f(0)[u,u]=0$. In particular, $P\nabla^3 f(0)[u,u]=0$.
\end{enumerate}
\end{lemma}

\begin{proof}
\ref{item:flat-cubic}.
Since $Hu=0$ and $f(0)=0$, the expansion $f(tu)=\frac{t^3}{6}\nabla^3 f(0)[u,u,u]+\cO(t^4)$ and Assumption~\ref{ass:quartic} give $f(tu)\ge m_0 t^4\norm{u}^4$. A nonzero cubic coefficient would 
force $f(tu)$ to be negative either for small positive or small negative values of $t$, thereby  contradicting $f\ge 0$.

\ref{item:convex-kills-b}.
Fix $w\in\R^d$. Define $\psi(t):=\ip{\nabla^2 f(tw)\,u}{u}$. Convexity near the origin gives $\nabla^2 f(tw)\succeq 0$ for small $t$, so $\psi(t)\ge 0$. Since $\psi(0)=\ip{Hu}{u}=0$, the point $t=0$ is a local minimizer of $\psi$, hence $\psi'(0)=\nabla^3 f(0)[w,u,u]=0$. Since $w$ is arbitrary, $\nabla^3 f(0)[u,u]=0$.
\end{proof}

\begin{lemma}[Gradient estimates]\label{lem:taylor}
For $\norm{x}$ sufficiently small, writing $x=u+v$ we have:
\begin{enumerate}[label=\textnormal{(\alph*)}, leftmargin=2.2em]
\item \label{item:Qg-diff} $Q\nabla f(x)=Q\nabla f(u)+O(\norm{u}\norm{v}+\norm{v}^2)$.
\item \label{item:flat-g} $\norm{Q\nabla f(u)}=\Theta(\norm{u}^3)$.
\end{enumerate}
If in addition $P\nabla^3 f(0)[u,u]=0$ for all $u\in\mathsf Q$ (which holds when $f$ is convex near the origin by \Cref{lem:vanishing}\ref{item:convex-kills-b}):
\begin{enumerate}[label=\textnormal{(\alph*)}, leftmargin=2.2em, start=3]
\item \label{item:basic-Pg} $P\nabla f(x)=Hv+O(\norm{u}^3+\norm{u}\norm{v}+\norm{v}^2)$.
\end{enumerate}
\end{lemma}

\begin{proof}
\ref{item:Qg-diff}.
Write $Q\nabla f(u+v)-Q\nabla f(u)=\int_0^1 Q\nabla^2 f(u+tv)\,v\,dt$. Since $QHv=0$ and $\|\nabla^2 f(z)-H\|_{\rm op}=O(\norm{z})$, the integrand is $O((\norm{u}+\norm{v})\norm{v})$ uniformly in $t$.

\ref{item:flat-g}.
By Lemma~\ref{lem:vanishing}\ref{item:flat-cubic}, the Taylor expansion of $f$ on $\mathsf Q$ starts at order four: $f(u)=p(u)+o(\norm{u}^4)$ where $p(u):=\frac{1}{24}\nabla^4 f(0)[u,u,u,u]$. Assumption~\ref{ass:quartic} gives $p(u)\ge \frac{m_0}{2}\norm{u}^4$ for small $\norm{u}$. Differentiating yields $$Q\nabla f(u)=Q\nabla p(u)+o(\norm{u}^3).$$ By Euler's identity, $\ip{Q\nabla p(u)}{u}=4p(u)\ge 2m_0\norm{u}^4$, so Cauchy--Schwarz gives $\norm{Q\nabla p(u)}\ge 2m_0\norm{u}^3$. For the upper bound, note that Lemma~\ref{lem:vanishing}\ref{item:flat-cubic} gives $\nabla^3 f(0)[u,u,u]=0$ for every $u\in\mathsf Q$. Since $\nabla^3 f(0)$ is symmetric and trilinear, replacing $u$ by $u+sw$ with $w\in\mathsf Q$ and extracting the coefficient of $s$ yields $\nabla^3 f(0)[w,u,u]=0$ for all $u,w\in\mathsf Q$. In particular, $\ip{Q\nabla^3 f(0)[u,u]}{w}=\nabla^3 f(0)[w,u,u]=0$ for every $w\in\mathsf Q$, so $Q\nabla^3 f(0)[u,u]=0$. The Taylor expansion of $Q\nabla f(u)$ therefore becomes $$Q\nabla f(u)=\underbrace{Q\nabla^2 f(0)[u]}_{=0}+\underbrace{Q\nabla^3 f(0)[u,u]}_{=0}+\frac16 Q\nabla^4 f(0)[u,u,u]+o(\norm{u}^3)=O(\norm{u}^3),$$
thereby completing the proof.

\ref{item:basic-Pg}.
The gradient expansion $\nabla f(x)=Hx+\frac12\nabla^3 f(0)[x,x]+\cO(\norm{x}^3)$ and $PH=H$ give $P\nabla f(x)-Hv=\frac12 P\nabla^3 f(0)[u+v,u+v]+\cO(\norm{x}^3)$. By Lemma~\ref{lem:vanishing}\ref{item:convex-kills-b}, the pure $u^2$ term vanishes; the remaining terms give the bound.
\end{proof}

\section{Proof of the main theorem}\label{sec:proof}

\begin{proof}[Proof of \Cref{thm:main}]
Set $\mu:=\lambda_{\min}(H|_{\mathsf P})$ and let $L$ denote a Lipschitz constant for $\nabla f$ near the origin. The proof has five steps.

\smallskip
\noindent\emph{Step 0: non-increase of distance.}
Since $0$ minimizes convex function $f$, we have $\ip{\nabla f(x)}{x}\ge f(x)\ge 0$ for all $x$ near $0$. Co-coercivity of $L$-smooth convex gradients gives $\ip{\nabla f(x)}{x}\ge \frac{1}{L}\norm{\nabla f(x)}^2$. For a gradient step $x^+:=x-\eta\nabla f(x)$ with $\eta\le 1/L$, we therefore have
\[
\norm{x^+}^2
=\norm{x}^2-2\eta\ip{\nabla f(x)}{x}+\eta^2\norm{\nabla f(x)}^2
\le \norm{x}^2-\eta\bigl(\tfrac{2}{L}-\eta\bigr)\norm{\nabla f(x)}^2
\le \norm{x}^2.
\]
For a Polyak step, the aiming inequality $\ip{\nabla f(x)}{x}\ge f(x)$ gives $\norm{x^+}^2\le \norm{x}^2-f(x)^2/\norm{\nabla f(x)}^2\le \norm{x}^2$. Hence every iterate of Algorithm~\ref{alg:main} satisfies $\norm{y}\le \norm{x_0}$, and all iterates remain in a neighborhood where \Cref{lem:taylor} applies. Choose $\eta_0\le \min\{1/L,\,1/\norm{H}\}$.

\smallskip
\noindent\emph{Step 1: Polyak contraction when $R(x)\ge\tau$.}
For any fixed $\tau>0$ with $\tau^{3/2}\sqrt{m_0}<1$, set $c_\star:=\tau^{3/2}\sqrt{m_0}$ and $q_P:=\sqrt{1-c_\star}\in(0,1)$. For any $x$ with $R(x)\ge\tau$, recall that the Polyak step satisfies
\[
\norm{x^+}^2
\le \norm{x}^2-\frac{f(x)^2}{\norm{\nabla f(x)}^2}.
\]
From the assumption $R(x)=f(x)/\norm{\nabla f(x)}^{4/3}\ge\tau$ we have $\norm{\nabla f(x)}^2\le (f(x)/\tau)^{3/2}$, hence
\[
\frac{f(x)^2}{\norm{\nabla f(x)}^2}
\ge \frac{f(x)^2}{(f(x)/\tau)^{3/2}}
=f(x)^{1/2}\tau^{3/2}.
\]
Assumption~\ref{ass:quartic} gives $f(x)\ge m_0\norm{x}^4$, so $f(x)^{1/2}\ge \sqrt{m_0}\,\norm{x}^2$ and therefore $f(x)^2/\norm{\nabla f(x)}^2\ge \tau^{3/2}\sqrt{m_0}\,\norm{x}^2$. We conclude the contraction $\norm{x^+}\le q_P\norm{x}$.

\smallskip
\noindent\emph{Step 2: projected-gradient contraction when $R(x)<\tau$.}
Set $G(x):=\norm{P\nabla f(x)}^2$. Since $f\in C^4$, the map $G$ is $C^2$ near the origin, so $\nabla G$ is locally Lipschitz with some constant $L_G$ on $\ball{\rho}{0}$. The descent lemma for $G$ gives
\[
G(x^+)\le G(x)-\eta\ip{\nabla G(x)}{\nabla f(x)}+\tfrac{L_G}{2}\eta^2\norm{\nabla f(x)}^2.
\]
Write $g:=\nabla f(x)=Pg+Qg$. Since $\nabla G(x)=2\nabla^2 f(x)\,Pg$ and $\|\nabla^2 f(x)-H\|_{\rm op}=O(\norm{x})$, we deduce
\[
\ip{\nabla G(x)}{g}
=2\ip{\nabla^2 f(x)\,Pg}{Pg}+2\ip{\nabla^2 f(x)\,Pg}{Qg}
\ge \mu\norm{Pg}^2-C_1\norm{Pg}\norm{Qg},
\]
where the first term uses $\ip{Hw}{w}\ge\mu\norm{w}^2$ for $w\in\mathsf P$ (absorbing the $O(\norm{x})$ error for $\rho$ small) and $C_1:=2\sup_{\ball{\rho}{0}}\norm{\nabla^2 f}$. By Young's inequality, $C_1\norm{Pg}\norm{Qg}\le \frac{\mu}{4}\norm{Pg}^2+\frac{C_1^2}{\mu}\norm{Qg}^2$, so
\[
\ip{\nabla G(x)}{g}\ge \tfrac{3\mu}{4}\norm{Pg}^2-\tfrac{C_1^2}{\mu}\norm{Qg}^2.
\]
Substituting back and using $\norm{g}^2=\norm{Pg}^2+\norm{Qg}^2$ yields:
\begin{equation}\label{eq:G-contract}
G(x^+)\le (1-c_G\eta)\,G(x)+C_G\eta\,\norm{Q\nabla f(x)}^2
\end{equation}
with $c_G:=\mu/2$ and $C_G:=2C_1^2/\mu$, valid for $\eta\le \mu/(2L_G)$.

Now choose $\delta>0$ so that $C_G\delta^{-2}\le c_G/2$. We now show the following claim.

\smallskip
\noindent\emph{Claim:} There exists $\kappa>0$ such that the condition $\norm{P\nabla f(x)}\le\delta\norm{Q\nabla f(x)}$ implies $R(x)\ge\kappa$, for all $\norm{x}$ sufficiently small.

\smallskip
\noindent\emph{Proof of Claim.}
Recall that Lemma~\ref{lem:taylor}\ref{item:basic-Pg} shows $P\nabla f(x)=Hv+O(\norm{u}^3+\norm{u}\norm{v}+\norm{v}^2)$, so
\[
\norm{Hv}\le \norm{P\nabla f(x)}+O(\norm{u}^3+\norm{u}\norm{v}+\norm{v}^2).
\]
Since we have $\norm{Hv}\ge\mu\norm{v}$ for $v\in\mathsf P$ and $\norm{P\nabla f(x)}\le\delta\norm{Q\nabla f(x)}=O(\norm{u}^3+\norm{u}\norm{v}+\norm{v}^2)$ (Lemma~\ref{lem:taylor}\ref{item:Qg-diff},\ref{item:flat-g}), we obtain $\mu\norm{v}\le O(\norm{u}^3+\norm{u}\norm{v}+\norm{v}^2)$. For $\norm{x}$ small enough the terms $\norm{u}\norm{v}$ and $\norm{v}^2$ can be absorbed into the left side, yielding $\norm{v}=O(\norm{u}^3)$.

In this regime, Lemma~\ref{lem:taylor}\ref{item:Qg-diff},\ref{item:flat-g} give
\[
\norm{\nabla f(x)}
\le \norm{Q\nabla f(x)}+\norm{P\nabla f(x)}
\le (1+\delta)\norm{Q\nabla f(x)}
=O(\norm{u}^3),
\]
while Assumption~\ref{ass:quartic} gives $f(x)\ge m_0\norm{x}^4\ge m_0\norm{u}^4$. Therefore, we conclude
\[
R(x)=\frac{f(x)}{\norm{\nabla f(x)}^{4/3}}
\ge \frac{m_0\norm{u}^4}{\Omega(\norm{u}^4)}
=\Omega(1).
\]
Denoting this positive lower bound by $\kappa$, completes the proof of the claim. \qed

\smallskip
Set $\tau_\star:=\min\{\kappa/2,\,(4m_0)^{-1/3}\}$, so that $\tau_\star^{3/2}\sqrt{m_0}\le 1/2<1$; this determines the trigger threshold in Algorithm~\ref{alg:main}. By contrapositive of the claim, $R(x)<\tau_\star\le\kappa$ implies $\norm{P\nabla f(x)}>\delta\norm{Q\nabla f(x)}$. Substituting into \eqref{eq:G-contract} and using $C_G\delta^{-2}\le c_G/2$ yields the estimate
\begin{equation}\label{eq:G-contract-clean}
G(x^+)\le (1-c_\sharp\eta)\,G(x),
\qquad c_\sharp:=c_G/2.
\end{equation}

\smallskip
\noindent\emph{Step 3: radius contraction within a virtual block.}
By Step~0, we know that $\norm{x_k}$ is non-increasing. We now show: starting from any $x_0$ with $r:=\norm{x_0}\le\rho$, within $N=O(\eta^{-1}\log(1/r))$ steps some iterate satisfies $\norm{x_k}\le q\,r$ for a fixed $q\in(0,1)$.

If a Polyak step fires at any step $k<N$, Step~1 gives $\norm{x_{k+1}}\le q_P\,r$, and the monotonicity from Step~0 keeps $\norm{x_m}\le q_P\,r$ for all $m>k$.

Otherwise, $R(x_k)<\tau_\star$ for $k=0,\dots,N-1$, so all $N$ steps are gradient steps. Iterating \eqref{eq:G-contract-clean}:
\begin{equation}\label{eq:endpoint-Pg}
\norm{P\nabla f(x_N)}\le e^{-c_\sharp\eta N/2}\,\norm{P\nabla f(x_0)}=O(e^{-c_\sharp\eta N/2}\,r).
\end{equation}
Choosing the implicit constant in $N=O(\eta^{-1}\log(1/r))$ strictly larger than $4/c_\sharp$ gives $\norm{P\nabla f(x_N)}\le Cr^{3+\alpha}$ for some $\alpha>0$ and $C>0$, so $\norm{P\nabla f(x_N)}\le \sigma r^3$ for any prescribed $\sigma>0$ once $\rho$ is small enough.

From Lemma~\ref{lem:taylor}\ref{item:basic-Pg} we have the estimate $Hv_N=P\nabla f(x_N)+O(\norm{u_N}^3+\norm{u_N}\norm{v_N}+\norm{v_N}^2)$. Since $\norm{P\nabla f(x_N)}=O(r^3)$ and $\norm{x_N}\le r$, absorbing lower-order terms gives $\norm{v_N}=O(r^3)$.

We now bound the drift of the $\mathsf Q$-component over the $N$ gradient steps. Since each step is $x_{k+1}=x_k-\eta\nabla f(x_k)$, projecting onto $\mathsf Q$ gives $u_{k+1}=u_k-\eta\,Q\nabla f(x_k)$, so $\norm{u_{k+1}-u_k}=\eta\norm{Q\nabla f(x_k)}$. Next to  bound $\norm{Q\nabla f(x_k)}$, observe by Lemma~\ref{lem:taylor}\ref{item:Qg-diff}, we have $\norm{Q\nabla f(x_k)}\le\norm{Q\nabla f(u_k)}+O(\norm{u_k}\norm{v_k}+\norm{v_k}^2)$, and by Lemma~\ref{lem:taylor}\ref{item:flat-g}, we have $\norm{Q\nabla f(u_k)}=O(\norm{u_k}^3)$. Since $\norm{x_k}\le r$ (Step~0), both $\norm{u_k}\le r$ and $\norm{v_k}\le r$, so $\norm{Q\nabla f(x_k)}=O(r^3+r^2)=O(r^2)$. Summing the per-step drift over $N=O(\eta^{-1}\log(1/r))$ steps yields:
\[
\norm{u_N-u_0}\le \sum_{k=0}^{N-1}\norm{u_{k+1}-u_k}=\eta\sum_{k=0}^{N-1}\norm{Q\nabla f(x_k)}=O(\eta N r^2)=O(r^2\log(1/r))=o(r).
\]
\noindent Next, fix a small value $\theta>0$  and set $q_S:=3\theta$. We consider two cases.

\paragraph{Case 1: $\norm{u_0}\le \theta r$.}
Then $\norm{u_N}\le 2\theta r$ for $\rho$ small. Recalling $\norm{v_N}=O(r^3)$, we deduce $\norm{x_N}\le q_S\,r$.

\paragraph{Case 2: $\norm{u_0}>\theta r$.}
Since $\norm{u_0}>\theta r$ and the total drift satisfies $\norm{u_N-u_0}=o(r)$, for $\rho$ small enough we have $\norm{u_N}\ge\theta r/2$.

We now show that $\norm{Q\nabla f(x_N)}=\Omega(r^3)$. To this end, by Lemma~\ref{lem:taylor}\ref{item:Qg-diff}, we have
\[
\norm{Q\nabla f(x_N)}\ge \norm{Q\nabla f(u_N)}-O(\norm{u_N}\norm{v_N}+\norm{v_N}^2).
\]
Since $\norm{v_N}=O(r^3)$, the error term is $O(r\cdot r^3+r^6)=O(r^4)$. By Lemma~\ref{lem:taylor}\ref{item:flat-g}, we have $\norm{Q\nabla f(u_N)}=\Theta(\norm{u_N}^3)\ge c\,(\theta r/2)^3=\Omega(r^3)$ for some constant $c>0$. For $\rho$ small, the $O(r^4)$ error is negligible compared to the $\Omega(r^3)$ leading term, so $\norm{Q\nabla f(x_N)}=\Omega(r^3)$, as claimed.

Now recall from \eqref{eq:endpoint-Pg} that we can make $\norm{P\nabla f(x_N)}\le \sigma r^3$ for any prescribed $\sigma>0$ by enlarging the constant in $N$. Choose $\sigma$ small enough that $\sigma r^3\le \delta\norm{Q\nabla f(x_N)}$, which is possible since $\norm{Q\nabla f(x_N)}=\Omega(r^3)$. Then $\norm{P\nabla f(x_N)}\le\delta\norm{Q\nabla f(x_N)}$, so the Claim from Step~2 gives $R(x_N)\ge\kappa\ge\tau_\star$. Therefore the next step is a Polyak step, and Step~1 yields $\norm{x_{N+1}}\le q_P\,r$.

\medskip
Set $q:=\max\{q_S,q_P\}\in(0,1)$. In both cases, $\norm{x_k}\le q\,r$ for some $k\le N+1$, and monotonicity ensures $\norm{x_m}\le q\,r$ for all $m\ge k$.

\smallskip
\noindent\emph{Step 4: complexity.}
Define virtual epochs for the analysis: epoch~$j$ starts at the first index $n_j$ with $\norm{x_{n_j}}\le r_j:=q^j\norm{x_0}$. If $\norm{x_{n_j}}\le q\,r_j=r_{j+1}$, then $n_{j+1}=n_j$ and epoch~$j$ uses zero steps. Otherwise $q\,r_j<\norm{x_{n_j}}\le r_j$, and Step~3 applied with start radius $\norm{x_{n_j}}$ ends the epoch within $O(\eta^{-1}\log(1/\norm{x_{n_j}}))=O(\eta^{-1}\log(1/r_j))$ steps (since $\norm{x_{n_j}}\ge q\,r_j$). To reach $\ball{\varepsilon}{0}$ requires $J=O(\log(\norm{x_0}/\varepsilon))$ epochs. The total step count is
\[
\sum_{j=0}^{J-1}N_j
=O\!\left(\eta^{-1}\sum_{j=0}^{J-1}\log\frac{1}{r_j}\right)
=O\!\left(\eta^{-1}\sum_{j=0}^{J-1}\bigl(j+\log\tfrac{1}{\norm{x_0}}\bigr)\right)
=\cO\!\left(\eta^{-1}\log\frac{1}{\varepsilon}\log\frac{\norm{x_0}}{\varepsilon}\right).
\qedhere
\]
\end{proof}

\section{Numerical illustrations}

\Cref{fig:experiments} compares Algorithm~\ref{alg:main} with fixed-step gradient descent, pure Polyak, and the block-scheduled \texttt{GDPolyak} from~\cite{DDJ} on three fourth-order instances from their repository: quartic Rosenbrock, overparameterized quadratic sensing, and the overparameterized single-neuron loss. Initializations match the repository; all hyperparameters---$(\eta,\tau)$ for Algorithm~\ref{alg:main}, the GD stepsize and block length for \texttt{GDPolyak}, and the GD stepsize---were tuned by grid search per instance. Each run stops when a problem-specific diagnostic crosses a threshold: Euclidean distance $<10^{-7}$ (Rosenbrock), Procrustes distance $<10^{-5}$ (quadratic sensing), or the DDJ optimality surrogate $<10^{-12}$ (single neuron). The iteration budget for each row is set so that \texttt{GDPolyak} reaches its threshold. With $(\eta,\tau)=(0.05,0.01)$, $(0.075,0.15)$, and $(1,0.0125)$, Algorithm~\ref{alg:main} reaches the threshold in $605$, $5418$, and $115$ iterations, compared with $2550$, $11055$, and $320$ for tuned \texttt{GDPolyak}.

\begin{figure}[t]
\centering
\includegraphics[width=\textwidth]{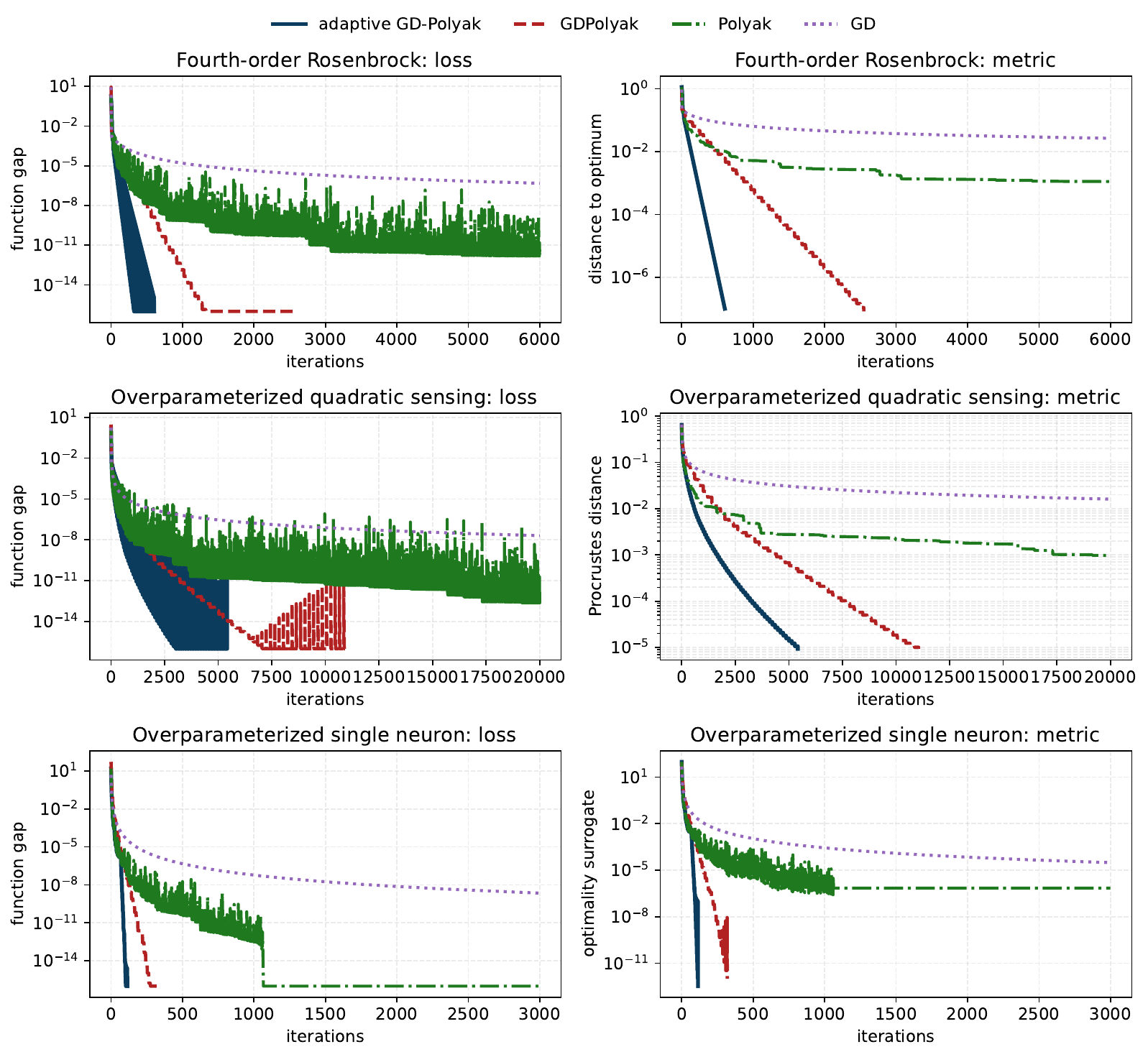}
\caption{DDJ fourth-order benchmarks. Rows: Rosenbrock (distance $10^{-7}$), quadratic sensing (Procrustes $10^{-5}$), single neuron (surrogate $10^{-12}$). Left: $f(x_k)-f_\star$. Right: stop diagnostic. Algorithm~\ref{alg:main}: $(\eta,\tau)=(0.05,0.01)$, $(0.075,0.15)$, $(1,0.0125)$. \texttt{GDPolyak}: (stepsize, block)${}=(0.03,50)$, $(0.075,200)$, $(1,10)$. GD stepsize: $0.03$, $0.075$, $1.5$.}
\label{fig:experiments}
\end{figure}

\bibliographystyle{unsrt}
\bibliography{convex_gd_polyak_note_refs}

@article{DDJ,
  author  = {Davis, Damek and Drusvyatskiy, Dmitriy and Jiang, Liwei},
  title   = {Gradient descent with adaptive stepsize converges (nearly) linearly under fourth-order growth},
  journal = {Math. Program.},
  year    = {2025},
  note    = {doi:10.1007/s10107-025-02290-5}
}

@article{Polyak63,
  author  = {Polyak, B. T.},
  title   = {Gradient methods for minimizing functionals},
  journal = {Zh. Vychisl. Mat. i Mat. Fiz.},
  volume  = {3},
  number  = {4},
  pages   = {643--653},
  year    = {1963}
}

@article{LuoTseng93,
  author  = {Luo, Z.-Q. and Tseng, P.},
  title   = {Error bounds and convergence analysis of feasible descent methods: a general approach},
  journal = {Ann. Oper. Res.},
  volume  = {46},
  number  = {1},
  pages   = {157--178},
  year    = {1993}
}

@article{DrusvyatskiyLewis18,
  author  = {Drusvyatskiy, D. and Lewis, A. S.},
  title   = {Error bounds, quadratic growth, and linear convergence of proximal methods},
  journal = {Math. Oper. Res.},
  volume  = {43},
  number  = {3},
  pages   = {919--948},
  year    = {2018}
}

@inproceedings{KarimiNutiniSchmidt16,
  author    = {Karimi, H. and Nutini, J. and Schmidt, M.},
  title     = {Linear convergence of gradient and proximal-gradient methods under the {P}olyak--{{\L}ojasiewicz} condition},
  booktitle = {Proceedings of ECML PKDD},
  pages     = {795--811},
  publisher = {Springer},
  year      = {2016}
}

@article{NecNesGli19,
  author  = {Necoara, I. and Nesterov, Yu. and Glineur, F.},
  title   = {Linear convergence of first order methods for non-strongly convex optimization},
  journal = {Math. Program.},
  volume  = {175},
  pages   = {69--107},
  year    = {2019}
}

@book{Polyak87,
  author    = {Polyak, B. T.},
  title     = {Introduction to Optimization},
  publisher = {Optimization Software, Inc.},
  address   = {New York},
  year      = {1987}
}

@misc{HazanKakade19,
  author = {Hazan, E. and Kakade, S.},
  title  = {Revisiting the {P}olyak step size},
  note   = {arXiv:1905.00313},
  year   = {2019}
}

@article{AttouchBolteSvaiter13,
  author  = {Attouch, H. and Bolte, J. and Svaiter, B. F.},
  title   = {Convergence of descent methods for semi-algebraic and tame problems: proximal algorithms, forward--backward splitting, and regularized {G}auss--{S}eidel methods},
  journal = {Math. Program.},
  volume  = {137},
  number  = {1},
  pages   = {91--129},
  year    = {2013}
}

@article{GelfandTsetlin61,
  author  = {Gelfand, I. M. and Tsetlin, M. L.},
  title   = {The principle of non-local search in automatic optimization systems},
  journal = {Dokl. Akad. Nauk SSSR},
  volume  = {137},
  pages   = {295--298},
  year    = {1961}
}

@article{Polyak64,
  author  = {Polyak, B. T.},
  title   = {Some methods of speeding up the convergence of iteration methods},
  journal = {USSR Comput. Math. Math. Phys.},
  volume  = {4},
  number  = {5},
  pages   = {1--17},
  year    = {1964}
}

@article{Nesterov83,
  author  = {Nesterov, Yu.},
  title   = {A method for solving the convex programming problem with convergence rate {$O(1/k^2)$}},
  journal = {Dokl. Akad. Nauk SSSR},
  volume  = {269},
  number  = {3},
  pages   = {543--547},
  year    = {1983}
}

@article{AttouchFadili22,
  author  = {Attouch, H. and Fadili, J.},
  title   = {From the ravine method to the {N}esterov method and vice versa: a dynamical system perspective},
  journal = {SIAM J. Optim.},
  volume  = {32},
  number  = {3},
  pages   = {2074--2101},
  year    = {2022}
}

@article{Lewis02,
  author  = {Lewis, A. S.},
  title   = {Active sets, nonsmoothness, and sensitivity},
  journal = {SIAM J. Optim.},
  volume  = {13},
  number  = {3},
  pages   = {702--725},
  year    = {2002}
}

@article{DrusvyatskiyLewis14,
  author  = {Drusvyatskiy, D. and Lewis, A. S.},
  title   = {Optimality, identifiability, and sensitivity},
  journal = {Math. Program.},
  volume  = {147},
  number  = {1},
  pages   = {467--498},
  year    = {2014}
}

@article{DavisJiang24,
  author  = {Davis, D. and Jiang, L.},
  title   = {A local nearly linearly convergent first-order method for nonsmooth functions with quadratic growth},
  journal = {Found. Comput. Math.},
  pages   = {1--82},
  year    = {2024}
}

@article{Young53,
  author  = {Young, D.},
  title   = {On {R}ichardson's method for solving linear systems with positive definite matrices},
  journal = {J. Math. Phys.},
  volume  = {32},
  number  = {1--4},
  pages   = {243--255},
  year    = {1953}
}

@inproceedings{KelnerEtAl22,
  author    = {Kelner, Jonathan and Marsden, Annie and Sharan, Vatsal and Sidford, Aaron and Valiant, Gregory and Yuan, Honglin},
  title     = {Big-Step-Little-Step: Efficient Gradient Methods for Objectives with Multiple Scales},
  booktitle = {Proceedings of Thirty Fifth Conference on Learning Theory},
  series    = {Proc. Mach. Learn. Res.},
  volume    = {178},
  pages     = {2431--2540},
  publisher = {PMLR},
  year      = {2022}
}

@article{Grimmer24,
  author  = {Grimmer, B.},
  title   = {Provably faster gradient descent via long steps},
  journal = {SIAM J. Optim.},
  volume  = {34},
  number  = {3},
  pages   = {2588--2608},
  year    = {2024}
}

@article{AltschulerParrilo23a,
  author  = {Altschuler, Jason M. and Parrilo, Pablo A.},
  title   = {Acceleration by stepsize hedging {I}: Multi-Step Descent and the Silver Stepsize Schedule},
  journal = {J. ACM},
  volume  = {72},
  number  = {2},
  pages   = {12:1--12:38},
  year    = {2025}
}

@article{AltschulerParrilo23b,
  author  = {Altschuler, Jason M. and Parrilo, Pablo A.},
  title   = {Acceleration by stepsize hedging {II}: Silver Stepsize Schedule for smooth convex optimization},
  journal = {Math. Program.},
  volume  = {213},
  number  = {1--2},
  pages   = {1105--1118},
  year    = {2025}
}

@article{BurerMonteiro03,
  author  = {Burer, S. and Monteiro, R. D. C.},
  title   = {A nonlinear programming algorithm for solving semidefinite programs via low-rank factorization},
  journal = {Math. Program.},
  volume  = {95},
  number  = {2},
  pages   = {329--357},
  year    = {2003}
}

@article{BurerMonteiro05,
  author  = {Burer, S. and Monteiro, R. D. C.},
  title   = {Local minima and convergence in low-rank semidefinite programming},
  journal = {Math. Program.},
  volume  = {103},
  number  = {3},
  pages   = {427--444},
  year    = {2005}
}

@article{ChiLuChen19,
  author  = {Chi, Y. and Lu, Y. M. and Chen, Y.},
  title   = {Nonconvex optimization meets low-rank matrix factorization: An overview},
  journal = {IEEE Trans. Signal Process.},
  volume  = {67},
  number  = {20},
  pages   = {5239--5269},
  year    = {2019}
}

@article{ZhuoEtAl21,
  author  = {Zhuo, J. and Kwon, J. and Ho, N. and Caramanis, C.},
  title   = {On the computational and statistical complexity of over-parameterized matrix sensing},
  journal = {J. Mach. Learn. Res.},
  volume  = {25},
  number  = {169},
  pages   = {1--47},
  year    = {2024}
}

@inproceedings{XuDu23,
  author    = {Xu, W. and Du, S. S.},
  title     = {Over-Parameterization Exponentially Slows Down Gradient Descent for Learning a Single Neuron},
  booktitle = {Proceedings of Thirty Sixth Conference on Learning Theory},
  series    = {Proc. Mach. Learn. Res.},
  volume    = {195},
  pages     = {1155--1198},
  publisher = {PMLR},
  year      = {2023}
}

@inproceedings{TuEtAl16,
  author    = {Tu, Stephen and Boczar, Ross and Simchowitz, Max and Soltanolkotabi, Mahdi and Recht, Ben},
  title     = {Low-rank Solutions of Linear Matrix Equations via {P}rocrustes Flow},
  booktitle = {Proceedings of The 33rd International Conference on Machine Learning},
  series    = {Proc. Mach. Learn. Res.},
  volume    = {48},
  pages     = {964--973},
  publisher = {PMLR},
  year      = {2016}
}

@inproceedings{GeJinZheng17,
  author    = {Ge, Rong and Jin, Chi and Zheng, Yi},
  title     = {No Spurious Local Minima in Nonconvex Low Rank Problems: A Unified Geometric Analysis},
  booktitle = {Proceedings of the 34th International Conference on Machine Learning},
  series    = {Proc. Mach. Learn. Res.},
  volume    = {70},
  pages     = {1233--1242},
  publisher = {PMLR},
  year      = {2017}
}

@article{ZhuEtAl18,
  author  = {Zhu, Z. and Li, Q. and Tang, G. and Wakin, M. B.},
  title   = {Global optimality in low-rank matrix optimization},
  journal = {IEEE Trans. Signal Process.},
  volume  = {66},
  number  = {13},
  pages   = {3614--3628},
  year    = {2018}
}

\end{document}